\documentclass[12pt,oneside]{amsart}
\usepackage{amsfonts}
\usepackage{amsmath}
\usepackage{amssymb}
\usepackage{amsthm}
\usepackage{mathrsfs}
\usepackage[bookmarksnumbered]{hyperref}
\usepackage[all]{xy}
 \SelectTips{cm}{}

\title{Rational equivariant rigidity}

\author{David Barnes}
\author{Constanze Roitzheim}
\address{C. Roitzheim \\ University of Glasgow \\ Department of Mathematics\\ University Gardens\\ Glasgow G12 8QW,
UK}
\address{D. Barnes\\ Department of Pure Mathematics\\
University of Sheffield \\ Hicks Building\\ Hounsfield Road\\
Sheffield S3 7RH, UK}
\thanks{The first author was supported by EPSRC grant EP/H026681/1, the second author by EPSRC grant EP/G051348/1. }

\date{$6^\text{th}$ January 2012}

\DeclareMathOperator{\Ho}{Ho}

\DeclareMathOperator{\Hom}{Hom}

\DeclareMathOperator{\Sp}{Sp}
\DeclareMathOperator{\Ch}{Ch}
\DeclareMathOperator{\h}{H}
\DeclareMathOperator{\classify}{B}
\DeclareMathOperator{\Dclassify}{DB}
\DeclareMathOperator{\E}{E}
\DeclareMathOperator{\DE}{DE}
\DeclareMathOperator{\leftmod}{mod--}

\newcommand{\GMQ}{G \Sp_\mathbb{Q}}
\newcommand{\TMQ}{\mathbb{T} \Sp_\mathbb{Q}}

\newtheorem{theorem}{Theorem}[section]
\newtheorem{proposition}[theorem]{Proposition}
\newtheorem{corollary}[theorem]{Corollary}
\newtheorem{lemma}[theorem]{Lemma}

\newtheorem{definition}[theorem]{Definition}
\newtheorem{ex}[theorem]{Example}

\newtheorem{thm}{Theorem}

\begin{document}
\begin{abstract}
We prove that if $G$ is $S^1$ or a profinite group,
then the all of the homotopical information of the 
category of rational $G$-spectra
is captured by triangulated structure of the
rational $G$-equivariant stable homotopy category. 
That is, for $G$ profinite or $S^1$, the 
rational $G$-equivariant stable homotopy category is rigid.
For the case of profinite groups this rigidity comes from 
an intrinsic formality statement, so we carefully relate 
the notion of intrinsic formality of a differential graded
algebra to rigidity. 
\end{abstract}
\maketitle

\section*{Introduction}

Quillen equivalences between stable model categories give rise to triangulated equivalences of their homotopy categories. The converse is not necessarily true - there are numerous examples of model categories that have equivalent homotopy categories but completely different homotopical behaviour. One example, see \cite[2.6]{SchShi02}, is module spectra over the Morava K-theories $K(n)$ and differential graded $K(n)_*$-modules. Their homotopy categories are equivalent triangulated categories, yet they cannot be Quillen equivalent as they have different mapping spaces.

As this converse statement is a very strong property, it is of great interest to find stable model categories $\mathcal{C}$ whose homotopical information is entirely determined by the triangulated structure of the homotopy category $\Ho(\mathcal{C})$. Such homotopy categories are called \emph{rigid}.

The first major result was found by Stefan Schwede who proved the rigidity of the stable homotopy category $\Ho(\Sp)$: any stable model category whose homotopy category is triangulated equivalent to $\Ho(\Sp)$ is Quillen equivalent to the model category of spectra \cite{Sch07}. To investigate further into the internal structure of the stable homotopy category, Bousfield localisations of the stable homotopy category have subsequently been considered. The second author showed in \cite{Roi07} that the $K$-local stable homotopy category at the prime 2 is rigid. Astonishingly, this is not true for odd primes, as a counterexample by Jens Franke shows in \cite{Fra96}, see also \cite{Roi08}.

The main focus of the proofs of the above results is the respective sphere spectra and their endomorphisms.  For both the stable homotopy category and its $p$-local $K$-theory localisation, the sphere is a \emph{compact generator}, meaning that it ``generates'' the entire homotopy category under exact triangles and coproducts. Studying the endomorphisms of a generator is essentially Morita theory. The idea is that all homotopical information of a model category can be deduced from a certain endomorphism ring object of its compact generators \cite{SchShi03}. In the case of a category with one compact generator, this endomorphism ring object is a symmetric ring spectrum.

The cases mentioned above all monogenic, that is, they have a single compact generator (the sphere spectrum). In this paper we are working with homotopy categories that have multiple generators. Our examples are taken from equivariant stable homotopy theory. This area of mathematics is of great general interest because of the prevalence of group actions in mathematics, so it makes good sense to use this area as our source of examples of non-monogenic rigidity. 

Specifically, the examples we study are model categories of rational $G$-equivariant spectra $\GMQ$, where $G$ is either a finite group, a profinite group or the circle group $S^1$. Recall that a profinite group is an inverse limit of an inverse system of finite groups, with the $p$-adic numbers being the canonical example. Any finite group is, of course, profinite, but whenever we talk of a profinite group we assume that the group is infinite. 

In the case of a finite group the category $\GMQ$ has been extensively studied by Greenlees, May and the first author \cite{Bar09}. For a profinite group the non-rationalised category is introduced and examined in \cite{fauskeqpro}. 
The first author studies the rationalised category in \cite{barnespadic}, 
with emphasis on the case of the $p$-adic integers. The case of $S^1$
has been studied in great detail by Greenlees and Shipley, \cite{GS11}.
The goal of this paper is to prove the following for $G$ finite, profinite or $S^1$.

\begin{thm}[\textbf{Rational $G$-equivariant Rigidity Theorem}] Let $\mathcal{C}$ be a proper, cofibrantly generated, simplicial, stable model category and let
\[
\Phi: \Ho(\GMQ) \longrightarrow \Ho(\mathcal{C})
\]
be an equivalence of triangulated categories. Then $\mathcal{C}$ and $\GMQ$ are Quillen equivalent.
\end{thm}

The homotopy category $\Ho(\GMQ)$ is not monogenic, a set of compact generators is, in the finite case, given by $\mathcal{G}_{top} = \{ \Sigma^\infty G/H_+ \}$, the suspension spectra of the homogeneous spaces $G/H$ for $H$ a subgroup of $G$. In the profinite case the generators are   $\mathcal{G}_{top} = \{ \Sigma^\infty G/H_+ \}$ as $H$ runs over the open subgroups of $G$. 
For $S^1$ the generators are $\mathcal{G}_{top} = \{ \Sigma^\infty S^1/H_+ \}$ as $H$ runs over the closed subgroups of $G$.
Hence, instead of studying an endomorphism ring object, we consider a ``ring spectrum with several objects'', which is a small spectral category. Via the above equivalence $\Phi$, $\mathcal{G}_{top}$ also provides a set of compact generators $\mathcal{X}=\Phi(\mathcal{E}_{top})$ for the homotopy category of $\mathcal{C}$, from which we can form its endomorphism category $\mathcal{E}(\mathcal{X})$. Generally, a triangulated equivalence on homotopy category would not be sufficient to extract enough information from $\mathcal{E}(\mathcal{X})$. However, in our case computations by Greenlees and May (for finite groups), the first author (for profinite groups) and Greenlees and Shipley (for $S^1$) allow us to construct a Quillen equivalence.

This theorem is particularly notable as it provides an example of rigidity in the case of multiple generators rather than just a monogenic homotopy category. We note that in the finite case, the coproduct of all the generators is also a compact generator, so technically speaking, $\GMQ$ can be thought of as monogenic in the finite case. In the profinite setting (where we have assumed the group to be infinite) or for $S^1$, there are countably many generators and no finite subset will suffice.  Furthermore, the coproduct of this infinite collection of generators will not be compact, so these model categories cannot be monogenic.

\subsection*{Organisation}
In Section \ref{modelcategories} we establish some notations and conventions before discussing the notions of generators and compactness.

Section \ref{Moritatheory} provides a summary of Schwede's and Shipley's Morita theory result which relates model categories to categories of modules over an endomorphism category of generators. More precisely, assume that $\mathcal{D}$ is a model category which satisfies some further minor technical assumptions and which has a set of generators $\mathcal{X}$. Then one can define the \emph{endomorphism category} $\mathcal{E}(\mathcal{X})$ and the model category of modules over it. Schwede and Shipley then give a sequence of Quillen equivalences
\[
\mathcal{D} \simeq_Q \leftmod \mathcal{E}(\mathcal{X}).
\]

In Section \ref{Gspectra} we recover some definitions and properties about rational $G$-spectra and describe its endomorphism category $\mathcal{E}(\mathcal{G}_{top}).$

In Section \ref{Quillenequivalence} we prove our rigidity theorem for finite and profinite groups. 
We begin by assuming that $\mathcal{C}$ is a well behaved model category with a triangulated equivalence between $\Ho(\mathcal{C})$ and $\Ho(\GMQ)$. Then when $G$ is a finite or profinite group, 
the previous three sections provide enough information to produce a Quillen equivalence between the given model category $\mathcal{C}$ and $\GMQ$. In detail, we use the Morita theory of Section \ref{Moritatheory} to obtain a Quillen equivalence
\[
\leftmod \mathcal{E}(\mathcal{G}_{top})
 \,\,\raisebox{-0.1\height}{$\overrightarrow{\longleftarrow}$}\,\,
\GMQ
\]
and a zig-zag of Quillen equivalences
between $\mathcal{C}$ and 
$\leftmod \mathcal{E}(\mathcal{X})$.
We then use the computations of Section \ref{Gspectra} to produce a series of Quillen equivalences relating
$\leftmod \mathcal{E}(\mathcal{G}_{top})$ and $\leftmod \mathcal{E}(\mathcal{X})$.

With the profinite and finite case complete, we turn to the rigidity statement for the circle group $S^1$ in Section \ref{sec:circlecase}. 

In Section \ref{sec:formality}, 
we introduce the notion of formality and give some examples. We then 
explain how our rigidity statements for profinite and finite groups
fit into this framework, but the $S^1$ case does not.

\section{Stable model categories and generators}\label{modelcategories}

We assume that the reader is familiar with the basics of Quillen model categories. We provide only a brief summary of the main notions in order to establish notation and other conventions.

A model category $\mathcal{C}$ is a category with three distinguished classes of morphisms denoted \emph{weak equivalences} $\xymatrix{\ar[r]^{\sim} & {} }$, \emph{fibrations} $\xymatrix{\ar@{>>}[r] &{} }$  and \emph{cofibrations}  $\xymatrix{  \ar@{>->}[r] & {} }$ satisfying some strong but rather natural axioms. A good reference is \cite{DwySpa95}. The main purpose of a model structure is enabling us to define a reasonable notion of homotopy between morphisms. One can then form the homotopy category $\Ho(\mathcal{C})$ of a model category $\mathcal{C}$ using homotopy classes of morphisms.

For a pointed model category $\mathcal{C}$ one can define a \emph{suspension functor} and \emph{loop functor} as follows. Let $X$ be an object in $\mathcal{C}$, without loss of generality let $X$ be fibrant and cofibrant. We then choose a factorisation
$$\xymatrix{ X \ar@{>->}[r] & C \ar[r]^{\sim}  & \ast }$$ of the map $X \longrightarrow *$. The suspension $\Sigma X$ is now defined as the pushout of the diagram $$\xymatrix{ \ast  & \ar[l]\ar@{>->}[r] X & C}.$$ The loop functor $\Omega$ is defined dually. Suspension and loop functors form an adjunction
\[
\Sigma: \Ho(\mathcal{C}) \,\,\raisebox{-0.1\height}{$\overrightarrow{\longleftarrow}$}\,\, \Ho(\mathcal{C}): \Omega .
\]
Note that when we write down an adjunction of functors the top arrow always denotes the left adjoint.

In the case of pointed topological spaces this recovers the usual suspension and loop functors. For the derived category of an abelian category the suspension functor is the shift functor of degree $+1$ and the loop functor the shift functor of degree $-1$. So in the latter case suspension and loop functors are inverse equivalences of categories, which is not the case for topological spaces.

\begin{definition}
A pointed model category $\mathcal{C}$ is called \emph{stable} if
\[
\Sigma: \Ho(\mathcal{C}) \,\,\raisebox{-0.1\height}{$\overrightarrow{\longleftarrow}$}\,\, \Ho(\mathcal{C}): \Omega
\]
are inverse equivalences of categories.
\end{definition}

One reason why stable model categories are of interest is because their homotopy categories are triangulated, which gives us a wealth of additional structure to make use of. Examples of stable model categories are chain complexes of modules over a ring $R$ with either the projective or injective model structure \cite[2.3]{Hov99} or symmetric spectra $\Sp$ in the sense of \cite{HSS}.

We also need to consider functors respecting the model structures on categories.

\begin{definition}
Let $F: \mathcal{C} \,\,\raisebox{-0.1\height}{$\overrightarrow{\longleftarrow}$}\,\, \mathcal{D}: G$ be an adjoint functor pair. Then $(F,G)$ is called a \emph{Quillen adjunction} if $F$ preserves cofibrations and acyclic cofibrations, or equivalently, $G$ preserves fibrations and acyclic fibrations.
\end{definition}

Note that a Quillen functor pair induces an adjunction $$LF: \Ho(\mathcal{C}) \,\,\raisebox{-0.1\height}{$\overrightarrow{\longleftarrow}$}\,\, \Ho(\mathcal{D}): RG$$ \cite[1.3.10]{Hov99}. The functors $LF$ and $RG$ are called the \emph{derived functors} of the Quillen functors $F$ and $G$. If the adjunction $(LF, RG)$ provides an equivalence of categories, then $(F,G)$ is called \emph{Quillen equivalence}. But Quillen equivalences do not only induce equivalences of homotopy categories, they also give rise to equivalences on all `higher homotopy constructions' on $\mathcal{C}$ and $\mathcal{D}$. To summarise, Quillen equivalent model categories have the same homotopy theory.

In the case of $\mathcal{C}$ and $\mathcal{D}$ being stable model categories, the derived functors $LF$ and $RG$ are \emph{exact} functors. This means that they respect the triangulated structures. For a triangulated category $\mathcal{T}$, we denote the morphisms in $\mathcal{T}$ by $[-,-]^{\mathcal{T}}_*$. In the case of $\mathcal{T}=\Ho(\mathcal{C})$, we abbreviate this to $[-,-]^{\mathcal{C}}_*$.

\begin{definition}
Let $\mathcal{T}$ be a triangulated category. A set $\mathcal{X} \subseteq \mathcal{T}$ is called \emph{a set of generators} if it detects isomorphisms, i.e. a morphism $f:A \longrightarrow B$ in $\mathcal{T}$ is an isomorphism if and only if
\[
f_*: [X,A]^{\mathcal{T}}_* \longrightarrow [X,B]^{\mathcal{T}}_*
\]
is an isomorphism for all $X \in \mathcal{X}$.
\end{definition}

\begin{definition}
If $\mathcal{T}$ is a triangulated category that has infinite coproducts, then
an object $Y \in \mathcal{T}$ is called \emph{compact} (or \emph{small}) if $[Y,-]^{\mathcal{T}}_*$ commutes with coproducts.
\end{definition}

The importance of those definitions can be seen in the following: if $\mathcal{T}$ has a set of compact generators $\mathcal{X}$, then any triangulated subcategory of $\mathcal{T}$ that contains $\mathcal{X}$ that is closed under coproducts must already be $\mathcal{T}$ itself, \cite[4.2]{Kel94}. Note that if $\Phi: \mathcal{T} \longrightarrow \mathcal{T}'$ is an equivalence of triangulated categories and $\mathcal{X} \subseteq \mathcal{T}$ a set of generators, it is immediate that $\Phi(\mathcal{X})$ is a set of generators in $\mathcal{T}'$.

\medskip
Examples of compact generators include the following.
\begin{itemize}
\item The sphere spectrum $S^0$ is a compact generator for the stable homotopy category $\Ho(\Sp)$.
\item Consider a smashing Bousfield localisation with respect to a homology theory $E_*$. Then the $E$-local sphere $L_E S^0$ is a compact generator of the $E$-local stable homotopy category $\Ho(L_E \Sp)$. However, if the localisation is not smashing, then $L_E S^0$ is a generator but not compact \cite[3.5.2]{HovPalStr97}.
\item Let $R$ be a commutative ring. Then the free $R$-module of rank one is a compact generator of the derived category $\mathcal{D}(R\mbox{-mod})$.
\end{itemize}

For a more detailed list of examples, \cite{SchShi03} is an excellent source, which also gives examples of non-monogenic triangulated categories such as $G$-spectra. The category of $G$-spectra will be discussed in detail in Section \ref{Gspectra}.

\section{Morita theory for stable model categories}\label{Moritatheory}

We are going to summarize some results and techniques of \cite{SchShi03} in this section. Schwede and Shipley show that, given a few minor technical assumptions, any stable model category is Quillen equivalent to a category of modules over an endomorphism ring object. In the case of a model category with a single generator $X$, this endomorphism object is a symmetric ring spectrum. However, in the case of a stable model category with a set of several generators such as rational $G$-spectra, the endomorphism object is a small category enriched over ring spectra, or a ``ring spectrum with several objects''.

We are going to assume that our stable model category $\mathcal{C}$ is a simplicial model category \cite[4.3]{Hov99} which is also proper and cofibrantly generated. As mentioned, being proper and cofibrantly generated are only minor technical assumptions which hold in most examples of reasonable model categories and being simplicial is less of a restriction than it may seem, via \cite{RezSchShi01} or \cite{Dug06}. In the latter reference, Dugger shows that any stable model category which is also ``presentable'' is Quillen equivalent to a spectral model category (defined below). These conditions (cofibrantly generated, proper and simplicial) will also be the assumptions of our main theorem in Section \ref{Quillenequivalence} later.

A \emph{spectral category} is a category $\mathcal{O}$ enriched, tensored and cotensored over symmetric spectra, see e.g. \cite[3.3.1]{SchShi03} or \cite[4.1.6]{Hov99}. A \emph{spectral model category} is a model category which is also a spectral category and further, the spectral structure is compatible with the model structure via the axiom (SP). The axiom (SP) is analogous to the axiom (SM7) that makes a simplicial category into a simplicial model category, see \cite[3.5.1]{SchShi03} or \cite[4.2.18]{Hov99}. A \emph{module over the spectral category $\mathcal{O}$} is a spectral functor
\[
M: \mathcal{O}^{op} \longrightarrow \Sp.
\]
A spectral functor consists of a symmetric spectrum $M(X)$ for each $X \in \mathcal{O}$ together with maps
\[
M(X) \wedge \mathcal{O}(Y,X) \longrightarrow M(Y)
\]
satisfying associativity and unit conditions, see \cite[3.3.1]{SchShi03}. The category $\leftmod\mathcal{O}$ of modules over the spectral category $\mathcal{O}$ can be given a model structure such that the weak equivalences are element-wise stable equivalences of symmetric spectra and fibrations are element-wise stable fibrations \cite[Theorem A.1.1]{SchShi03}.

To define the endomorphism category of a cofibrantly generated, simplicial, proper stable model category $\mathcal{C}$, we first have to replace $\mathcal{C}$ by a spectral category. In \cite[3.6]{SchShi03} Schwede and Shipley describe the category $\Sp(\mathcal{C})$ of symmetric spectra over $\mathcal{C}$, i.e. symmetric spectra with values in $\mathcal{C}$ rather than pointed simplicial sets. Theorem 3.8.2 of \cite{SchShi03} states how $\mathcal{C}$ can be replaced by $\Sp(\mathcal{C})$.

\begin{theorem}[Schwede-Shipley]\label{spectralreplacement}
The category $\Sp(\mathcal{C})$ can be given a model structure, the stable model structure, which makes $\Sp(\mathcal{C})$ into a spectral model category that is Quillen equivalent to $\mathcal{C}$ via the adjunction
\[
\Sigma^\infty: \mathcal{C}  \,\,\raisebox{-0.1\height}{$\overrightarrow{\longleftarrow}$}\,\, \Sp(\mathcal{C}): Ev_0.
\]
\end{theorem}

Now let $\mathcal{D}$ be a spectral model category with a set of compact generators $\mathcal{X}$. We define the \emph{endomorphism category} $\mathcal{E}(\mathcal{X})$ as having objects $X \in \mathcal{X}$ and morphisms $$\mathcal{E}(\mathcal{X})(X_1,X_2)=
\Hom_{\mathcal{D}}(X_1,X_2).$$ Here $\Hom_{\mathcal{D}}(-,-)$ denotes the homomorphism spectrum. This is an object in the category of symmetric spectra and comes as a part of the spectral enrichment of $\mathcal{D}$. The category $\mathcal{E}(\mathcal{X})$ is obviously a small spectral category. In the case of $\mathcal{X}=\{X \}$, $\mathcal{E}(\mathcal{X})(X,X)$ is a symmetric ring spectrum, the \emph{endomorphism ring spectrum of $X$}.

Without loss of generality we assume our generators to be both fibrant and cofibrant.  One can now define an adjunction
\[
-\wedge_{\mathcal{E}(\mathcal{X})} \mathcal{X} :
\leftmod\mathcal{E}(\mathcal{X})
\,\,\raisebox{-0.1\height}{$\overrightarrow{\longleftarrow}$}\,\,
\mathcal{D}   :
\Hom(\mathcal{X},-)
\]
where the right adjoint is given by $\Hom(\mathcal{X},Y)(X)=\Hom_{\mathcal{D}}(X,Y)$, \cite[3.9.1]{SchShi03}. The left adjoint is given via
\[
M \wedge_{\mathcal{E}(\mathcal{X})} \mathcal{X} = :eq \Big( \bigvee\limits_{X_1,X_2 \in \mathcal{X}} M(X_2) \wedge \mathcal{E}(\mathcal{X})(X_1,X_2)   \wedge X_1 \rightrightarrows \bigvee\limits_{X\in \mathcal{X}} M(X) \wedge X \Big).
\]
One map in the coequaliser is induced by the module structure, the other one comes from the evaluation map$$\mathcal{E}(\mathcal{X})(X_1,X_2)   \wedge X_1 \longrightarrow X_2.$$ Schwede and Shipley then continue to prove \cite[Theorem 3.9.3]{SchShi03} which says that in the case of all generators being compact, the above adjunction forms a Quillen equivalence. Combining this with Theorem \ref{spectralreplacement} one arrives at the following.

\begin{theorem}[Schwede-Shipley]\label{moritatheorem}
Let $\mathcal{C}$ be a proper, cofibrantly generated, simplicial,  stable model category with a set of compact generators $\mathcal{X}$. Then there is a zig-zag of simplicial Quillen equivalences
\[
\mathcal{C}
\simeq_Q
\leftmod\mathcal{E}(\mathcal{X}).
\]
\end{theorem}

Recall that if $\mathcal{C}$ is stable then for any $X$ and $Y$ in $\mathcal{C}$, 
$[X,Y]^\mathcal{C}$ is an abelian group. This group is said to be rational if it is 
uniquely divisible for any $n \in \mathbb{Z}$. This is equivalent to asking that 
the canonical map $[X,Y]^\mathcal{C} \to [X,Y]^\mathcal{C} \otimes \mathbb{Q}$ is an isomorphism. 
We can extend this to the whole of $\mathcal{C}$

\begin{definition}
A stable model category $\mathcal{C}$ is said to be \textbf{rational} if 
for any $X$ and $Y$ in $\mathcal{C}$ the set of maps in the homotopy category
$[X,Y]^\mathcal{C}$ is always rational.  
\end{definition}

When $\mathcal{C}$ satisfies our usual assumptions and is also rational, 
we can use the results of \cite{shi07} to show that $\mathcal{C}$ is 
Quillen equivalent to a $\Ch(\mathbb{Q})$-model category. This is analogous to the previously stated results concerning
 a spectral model category but using 
$\Ch(\mathbb{Q})$ rather than the model category of spectra.
Specifically we use the above theorem to see that 
$\mathcal{C}$ is Quillen equivalent to $\leftmod\mathcal{E}(\mathcal{X})$.
Since $\mathcal{C}$ is rational, we have an objectwise Quillen equivalence
of spectral categories between $\mathcal{E}(\mathcal{X})$ and 
$\mathcal{E}(\mathcal{X}) \wedge  H \mathbb{Q}$. The results of 
\cite{shi07} give us a series of Quillen equivalences that we can 
apply to $\mathcal{E}(\mathcal{X}) \wedge  H \mathbb{Q}$
to obtain a $\Ch(\mathbb{Q})$-model category 
$\mathcal{E}_\mathbb{Q}(\mathcal{X})$. Furthermore 
these functors provide us with a specified isomorphism between
$\h_* \mathcal{E}_\mathbb{Q}(\mathcal{X})$
and 
$\pi_* \mathcal{E}(\mathcal{X})$.
More details can be found in 
\cite[Section 6]{Bar09}.

\begin{theorem}[Shipley]\label{thm:rationalmorita}
Let $\mathcal{C}$ be a cofibrantly generated, simplicial, proper, rational,  
stable model category with a set of compact generators $\mathcal{X}$. Then there 
is a zig-zag of Quillen equivalences
\[
\mathcal{C}
\simeq_Q
\leftmod\mathcal{E}_\mathbb{Q}(\mathcal{X}).
\]
between $\mathcal{C}$ and a $\Ch(\mathbb{Q})$-model category 
$\mathcal{E}_\mathbb{Q}(\mathcal{X})$.
Furthermore there is a specified isomorphism 
of categories enriched in graded abelian groups
between
$\h_* \mathcal{E}_\mathbb{Q}(\mathcal{X})$
and 
$\pi_* \mathcal{E}(\mathcal{X})$.
\end{theorem}

\section{Rational \texorpdfstring{$G$}{G}-spectra and their endomorphism category}\label{Gspectra}

There are several Quillen equivalent constructions of spectra, each with their own various advantages and disadvantages. 
When one wants to consider $G$-spectra these differences become even more pronounced. 
We find that the most convenient construction is equivariant orthogonal spectra from \cite{mm02}, 
since the definitions can easily be generalised to the profinite case, as in \cite{fauskeqpro}. 
Recall that when we talk of a profinite group, we have assumed that group to be infinite.

Briefly, a $G$-spectrum $X$ consists of a collection of $G$-spaces $X(U)$, one for each finite dimensional real representation $U$ of $G$, with $G$-equivariant suspension maps $$S^V \wedge X(U) \to X(U \oplus V).$$ Here, $S^V$ is the one-point-compactification of the vector space $V$. A map of $G$-spectra is then a collection of $G$-maps $$f(U) : X(U) \to Y(U)$$ commuting with the suspension structure maps. An orthogonal $G$-spectrum has more structure still, but the underlying idea is the same. We denote the category of orthogonal $G$-spectra by $G\Sp$.

There are several model structures on $G \Sp$, which vary according to what subgroups of $G$ are of interest. 
We are concerned with all subgroups when $G$ is finite, only the closed subgroups when $G=S^1$ and 
only the open subgroups when $G$ is profinite. 
These choices are standard in topology, in each case the chosen collection is the one 
of most interest to topologists.  
Note that all subgroups of a finite group are closed and 
every open subgroup of a profinite group is closed. 
From now on we only talk of \emph{admissible} subgroups, taking it to mean 
any subgroup of a finite group, any closed subgroup of $S^1$ and any open 
subgroup of a profinite group. 

We thus have model structures on the categories of $G$-spectra for each type of group $G$. 
Following \cite{Bar09} these model categories can be rationalised, 
to form $L_\mathbb{Q} G\Sp$ which we denote as $\GMQ$.

\begin{theorem}
There is a model structure on $\GMQ$ such that the weak equivalences are those maps $f$ such that $\pi_*^H(f) \otimes \mathbb{Q}$ is an isomorphism for all admissible subgroups $H$ of $G$. This model category is proper, cofibrantly generated, monoidal and spectral.
\end{theorem}

The homotopy category $\Ho(\GMQ)$ has a finite set of compact generators in the case of a finite groups and a countable collection in the case of a profinite group or $S^1$.

\begin{lemma}\label{generators}
The fibrant replacements of the spectra $\Sigma^\infty G/H_+$ for $H$ an admissible subgroup form a set of compact generators denoted $\mathcal{G}_{top}$ for $\Ho(\GMQ)$.
\end{lemma}

For a proof of this, see e.g. \cite[Lemma 4.6]{fauskeqpro} for profinite groups or \cite[Section 2.1]{gre99} for the case of $S^1$.

The spectra $\Sigma^\infty G/H_+$ themselves are usually chosen to form $\mathcal{G}_{top}$, but for technical reasons we would like the generators to be fibrant and cofibrant. They are cofibrant to begin with and taking the fibrant replacement obviously does not change their property as generators. We denote these replacements by $\{ \Sigma^\infty_f G/H_+ \}$.

We now take a closer look at the endomorphism category $\mathcal{E}_{top}:=\mathcal{E}(\mathcal{G}_{top})$ of $\GMQ$. We remember from Theorem \ref{moritatheorem} that $\GMQ$ is Quillen equivalent to the category of modules $\leftmod\mathcal{E}_{top}$ over a spectral category $\mathcal{E}_{top}$. The category $\mathcal{E}_{top}$ has objects $\mathcal{G}_{top}=\{ \Sigma^\infty_f G/H_+ \}$. For $\sigma_1, \sigma_2 \in \mathcal{G}_{top}$, the morphisms are defined as $$\mathcal{E}_{top}(\sigma_1,\sigma_2)=\Hom_{\GMQ}(\sigma_1,\sigma_2).$$

We now introduce another bit of notation. Let $\sigma_1, \sigma_2 \in \mathcal{G}_{top}$ be two generators. Let $\underline{A}(\sigma_1,\sigma_2)$ denote $\pi_0(\mathcal{E}_{top}(\sigma_1,\sigma_2))$. Using the composition $$\mathcal{E}_{top}(\sigma_2,\sigma_3) \wedge \mathcal{E}_{top}(\sigma_1,\sigma_2) \longrightarrow \mathcal{E}_{top}(\sigma_1,\sigma_3)$$ we see that $\underline{A}$ forms a \emph{ringoid} or \emph{ring with several objects}, i.e. a category with objects $\mathcal{G}_{top}$ and morphisms $\underline{A}(\sigma_1,\sigma_2)$ together with composition maps $$\underline{A}(\sigma_2,\sigma_3) \otimes \underline{A}(\sigma_1,\sigma_2) \longrightarrow \underline{A}(\sigma_1,\sigma_3)$$ satisfying associativity and unital conditions.

Applying the Eilenberg-MacLane functor $H$ (see e.g. \cite{HSS} or \cite{Sch08}) then yields another spectral category $H\underline{A}$.

Let us return to $\mathcal{E}_{top}$, in \cite[Appendix A]{GreMay95}, Greenlees and May computed the groups $$[\Sigma^\infty G/H_+, \Sigma^\infty G/K_+]^{G\Sp}_* \otimes \mathbb{Q}$$ for subgroups $H$ and $K$ of a finite group $G$. Using the Segal-tom Dieck splitting result \cite[7.10]{fauskeqpro} one can similarly compute these groups in the case of a profinite group. In section 
\ref{sec:circlecase} we provide the analogous calculation for $S^1$, note that the following
theorem does not hold when $G=S^1$.

\begin{theorem}\label{homotopygroups}
Let $G$ be a finite or a profinite group. 
In degrees away from zero, $[\Sigma^\infty G/H_+, \Sigma^\infty G/K_+]^{G\Sp}_* $ is torsion. Hence, the homotopy groups of the spectrum $\mathcal{E}_{top}(\sigma_1,\sigma_2)$ are concentrated in degree zero.
\end{theorem}

It is not too surprising that a symmetric spectrum with homotopy groups concentrated in degree zero is weakly equivalent to an Eilenberg-MacLane spectrum. For a statement like this, see \cite[Theorem 4.22]{Sch07b}. However, we are also after a statement about the category of modules over this spectral category. Schwede and Shipley prove the following in \cite[Theorem A.1.1 and Proposition B.2.1]{SchShi03}

\begin{theorem}[Schwede-Shipley]\label{EMLmodules}
Let $\underline{R}$ be a spectral category whose morphism spectra are fibrant in $\Sp$ and have homotopy groups concentrated in degree zero. Then the module categories $\leftmod\underline{R}$ and $\leftmod \h \underline{\pi_0 R}$ are related by a chain of Quillen equivalences.
\end{theorem}

We can apply this theorem to the case of a finite or a profinite group. 
We have chosen our generators $\mathcal{G}_{top}$ to be fibrant and cofibrant. Hence $\mathcal{E}_{top}(\sigma_1,\sigma_2)=\Hom_\GMQ(\sigma_1,\sigma_2)$ is fibrant as for cofibrant $\sigma_1$, $\Hom_\GMQ(\sigma_1,-)$ is a right Quillen functor by definition.

\begin{corollary}
Let $G$ be a finite or a profinite group. 
The categories $\leftmod\mathcal{E}_{top}$ and $\leftmod \h \underline{A}$ are Quillen equivalent.
\end{corollary}

Combining this with Theorem \ref{moritatheorem} yields the following corollary which, in the finite case, is \cite[5.1.2]{SchShi03}. 
The first author has also considered monoidal structures for finite
groups or the $p$-adic integers in \cite{Bar09} and \cite{barnespadic}.

\begin{corollary}\label{GMQEML}
Let $G$ be a finite or a profinite group. 
There is a chain of Quillen equivalences between rational $G$-spectra $\GMQ$ and $\leftmod \h \underline{A}$.
\end{corollary}

\section{Rigidity for Finite and Profinite groups}\label{Quillenequivalence}

We are finally going to put together the results from the previous section to obtain our main theorem.

\begin{theorem}
Let $G$ be either a finite group or a profinite group. Let $\mathcal{C}$ be a cofibrantly generated, proper, simplicial, stable model category together with an equivalence of triangulated categories
\[
\Phi: \Ho(\GMQ) \longrightarrow \Ho(\mathcal{C}).
\]
Then $\GMQ$ and $\mathcal{C}$ are Quillen equivalent.
\end{theorem}

\begin{proof}
We showed in Corollary \ref{GMQEML} at the end of Section \ref{Gspectra} how $\GMQ$ is Quillen equivalent to the category of modules over the Eilenberg-MacLane ``ring spectrum of several objects'' 
$\leftmod \h \underline{A}$.

Let $\mathcal{X}$ consist of fibrant and cofibrant replacements of $\Phi(\sigma)$ where $\sigma \in \mathcal{G}_{top}$ runs over the generators of $\Ho(\GMQ)$. The set $\mathcal{X}$ is then a set of generators for $\Ho(\mathcal{C})$.
By Theorem \ref{moritatheorem} we have Quillen equivalences
\[
\GMQ \simeq_Q \leftmod\mathcal{E}_{top} \,\,\,\,\,\,\mbox{and} \,\,\,\,\,\, \mathcal{C} \simeq_Q \leftmod\mathcal{E}(\mathcal{X}).
\]
We are now going to compare $\leftmod\mathcal{E}_{top}$ to $\leftmod\mathcal{E}(\mathcal{X})$ by relating them both to  
$\leftmod \h \underline{A}$. Let $X_1$ be a cofibrant and fibrant replacement for $\Phi(\sigma_1)$ where $\sigma_1 \in \mathcal{G}_{top}$ and define $X_2$ analogously. Remember that
$\mathcal{E}(\mathcal{X})(X_1,X_2)$ was defined as the homomorphism spectrum of the fibrant replacement of the suspension spectrum of $X_1$ and $X_2$ in $\Sp(\mathcal{C})$ \cite[Definition 3.7.5]{SchShi03}, so
\[
\mathcal{E}(\mathcal{X})(X_1,X_2) = \Hom_{\Sp(\mathcal{C})}(\Sigma^\infty_{f} X_1, \Sigma^\infty_{f} X_2).
\]
By adjunction and using the Quillen equivalence $\Sigma^\infty: \mathcal{C}  \,\,\raisebox{-0.1\height}{$\overrightarrow{\longleftarrow}$}\,\, \Sp(\mathcal{C}): Ev_0$ we obtain
\[
[S^0, \mathcal{E}(\mathcal{X})(X_1,X_2)]^{\Sp}_* \cong [\Sigma^\infty X_1, \Sigma^\infty X_2]^{\Sp(\mathcal{C})}_*\cong [X_1, X_2]^{\mathcal{C}}_*.
\]
Via the equivalence $\Phi$ and again by adjunction this equals
\[
[\sigma_1,\sigma_2]^{\GMQ}_* \cong [S^0, \mathcal{E}_{top}(\sigma_1,\sigma_2)]^{\Sp}_*.
\]
Thus, $\mathcal{E}_{top}(\sigma_1,\sigma_2)$ and $\mathcal{E}(\mathcal{X})(X_1,X_2)$ have the same homotopy groups.
By Lemma \ref{homotopygroups}, these are concentrated in degree zero where they equal $\underline{A}(\sigma_1,\sigma_2)$. As the generators $X_i \in \mathcal{X}$ have been chosen to be fibrant and cofibrant, the spectra $\mathcal{E}(\mathcal{X})(X_1,X_2)$ are all fibrant in $\Sp$. Hence Theorem \ref{EMLmodules} applies, giving us a chain of Quillen equivalences between $\leftmod\mathcal{E}(\mathcal{X})$ and $\leftmod \h \underline{A}$.
Thus we have arrived at a collection of Quillen equivalences
\[
\GMQ \,\,\raisebox{-0.1\height}{$\overleftarrow{\longrightarrow}$}\,\,
\leftmod\mathcal{E}_{top} \,\,\simeq_Q\,\,
\leftmod \h \underline{A}  \,\,\simeq_Q\,\,
\leftmod\mathcal{E}(\mathcal{X})  \,\,
\raisebox{-0.1\height}{$\overrightarrow{\longleftarrow}$}\,\, \mathcal{C}
\]
which concludes our proof of the $G$-equivariant Rigidity Theorem.
\end{proof}

Hence we have presented a nontrivial example of rigidity that is not monogenic. It is a subject of further research whether rigidity also holds for $G$-spectra $G\Sp$ before rationalisation.

\section{\texorpdfstring{$S^1$}{S1}-equivariant rigidity}\label{sec:circlecase}

In the following part of this paper we are turning to the case 
where our group is $S^1=\mathbb{T}$, the fundamental example of an infinite compact Lie group.
In \cite{gre99}, John Greenlees constructed an abelian category $\mathcal{A}$ whose derived category $$D(\mathcal{A})=\Ho(\partial \mathcal{A})$$ is triangulated equivalent to $\Ho(\TMQ)$. In \cite{Shi02} Brooke Shipley used Morita theory to show that $D(\mathcal{A})$ and $\Ho(\TMQ)$ are in fact equivalent via a zig-zag of Quillen equivalences. We are going to see that Shipley's proof actually only depends on information which is encoded in the triangulated structure of $\Ho(\TMQ)$. Hence we will arrive at the conclusion that $\Ho(\TMQ)$ is indeed rigid.

\bigskip
Let us first turn to the algebraic model $\mathcal{A}$, which naturally has a grading, and the version with differentials, $\partial \mathcal{A}$. 

\begin{definition}
Let $\mathcal{F}$ be the set of finite subgroups of $\mathbb{T}$.
Let $\mathcal{O}_\mathcal{F}$ be the ring of operations 
$\prod_{H \in \mathcal{F}} \mathbb{Q} [c_H]$
with $c_H$ of degree $-2$. Let $e_H$ be the idempotent arising from 
projection onto factor $H$. 
In general, let $\phi$ be a subset of $\mathcal{F}$ and 
define $e_\phi$ to be the idempotent
coming from projection onto the factors in $\phi$. 
\end{definition}

We let $c$ be the sum of all elements $c_H$ for varying $H$. 
We can then write $c_H = e_H c$. 

\begin{definition}
For some function $v : \mathcal{F} \to \mathbb{Z}_{\geqslant 0}$ 
define $c^v \in \mathcal{O}_\mathcal{F}$ to be that element
with $e_H c^v = c_H^{v(H)}$. 
Let $\mathcal{E}$ be the set 
\[
\{ c^\nu \ | \ \nu : \mathcal{F} \to \mathbb{Z}_{\geqslant 0} \textrm{ with finite support}  \}.
\] 
We call this the set of \textbf{Euler classes}. 
\end{definition}
We form a new ring by formally adding inverses to the Euler classes, we call this ring 
$\mathcal{E}^{-1} \mathcal{O}_\mathcal{F}$. 
Note that this graded rational vector space is a ring, since the set of Euler classes is multiplicatively closed. 
To illustrate its structure, we see that 
as a graded vector space it is concentrated in even degrees, where  
\[
(\mathcal{E}^{-1} \mathcal{O}_\mathcal{F})_{2n} \cong \prod_{H \in \mathcal{F}} \mathbb{Q} \,\,\,\mbox{for}\,\,\,
n \leqslant 0 \,\,\,\mbox{and} \,\,\, 
(\mathcal{E}^{-1} \mathcal{O}_\mathcal{F})_{2n} \cong \oplus_{H \in \mathcal{F}} \mathbb{Q} \,\,\,\mbox{for}\,\,\, n >0.
\]

\begin{definition}
We define the category $\mathcal{A}$ to have object-class the 
collection of maps $$\beta : N \to \mathcal{E}^{-1} \mathcal{O}_\mathcal{F} \otimes V$$ of 
$\mathcal{O}_\mathcal{F}$-modules. Here, $V$ is a graded rational vector space, 
such that $\mathcal{E}^{-1} \beta$ is an isomorphism. 
The $\mathcal{O}_\mathcal{F}$-module $N$ is called the \emph{nub} and 
$V$ is called the \emph{vertex}.
A map $(\theta, \phi)$ in $\mathcal{A}$ is a commutative square 
$$
\xymatrix@C+1cm{
N \ar[r]^(0.4)\beta \ar[d]_\theta & 
\mathcal{E}^{-1} \mathcal{O}_\mathcal{F} \otimes V \ar[d]^{id \otimes \phi} \\
N' \ar[r]^(0.4){\beta'} &
\mathcal{E}^{-1} \mathcal{O}_\mathcal{F} \otimes V'
}$$
where $\theta$ is a map of $\mathcal{O}_\mathcal{F}$-modules and 
$\phi$ is a map of graded rational vector spaces. 
\end{definition}

If we think of $\mathcal{O}_\mathcal{F}$ as a chain complex with trivial 
differential, then we can consider the category of 
$\mathcal{O}_\mathcal{F}$-modules in $\Ch(\mathbb{Q})$. 
Such an object $N$ is an 
$\mathcal{O}_\mathcal{F}$-module in graded vector spaces along with maps 
$\partial_n : N_n \to N_{n-1}$. These maps
satisfy the relations
\[
\partial_{n-1} \circ \partial_n =0 
\quad
c \partial_n = \partial_{n-2} c.
\]

\begin{definition}
The category $\partial \mathcal{A}$ has object-class the 
collection of maps 
\[
\beta : N \to \mathcal{E}^{-1}\mathcal{O}_\mathcal{F} \otimes V
\]
of $\mathcal{O}_\mathcal{F}$-modules in $\Ch(\mathbb{Q})$. Here, 
$N$ is a rational chain complex with an action of $\mathcal{O}_\mathcal{F}$, 
$V$ is a rational chain complex and 
$\mathcal{E}^{-1} \beta$ is an isomorphism. 
A map $(\theta, \phi)$ is then a commutative square as before, such that 
both $\theta$ is a map in the category  
$\mathcal{O}_\mathcal{F}$-modules in $\Ch(\mathbb{Q})$ 
and $\phi$ is a map of $\Ch(\mathbb{Q})$. 
\end{definition}

This category has a set of compact generators, the so-called 
\emph{basic algebraic cells}, see \cite[Subsection 5.8]{gre99},
$$\mathcal{B}_a=\{ I_H\}_{H \leq \mathbb{T}}.$$ 
The $I_H$ are fibrant replacements of the image of the ``geometric basic cells'' in $\TMQ$ under Greenlees' triangulated equivalence, 
see \cite[Proposition 2.9]{Shi02}.  
We do not need to know the construction of those cells, only the properties
exploited in the work of Shipley. Hence, we do not give any further details of their construction.
By Theorem \ref{thm:rationalmorita} we know
\[
\partial \mathcal{A} \simeq_Q \leftmod\mathcal{E}_\mathbb{Q}(\mathcal{B}_a).
\]
However, the dga of several objects $\mathcal{E}_\mathbb{Q}(\mathcal{B}_a)$ is quite large and difficult to handle. Shipley proceeds to construct a smaller, quasi-isomorphic category $\mathcal{E}_a$ by first taking the $-1$-connected cover $\mathcal{E}_\mathbb{Q}(\mathcal{B}_a)\left<0\right>$ and manually killing some of its generators. Finally, it is shown that $\mathcal{E}_a$ is quasi-isomorphic to an even smaller dga $\mathcal{S}$ whose generators, differentials and multiplicative structure are specified in \cite[Definition 6.3]{Shi02}. Hence,
\[
\partial \mathcal{A} \simeq_Q \leftmod\mathcal{E}_a 
\simeq_Q \leftmod\mathcal{S}.
\]
Shipley computed the relevant homology data in \cite[Proposition 4.9 and Proposition 6.4]{Shi02}.

\begin{proposition}[Shipley]\label{prop:homology}
The homology of $\mathcal{E}(\mathcal{B}_a)$ (and hence of $\mathcal{E}_a$ and $\mathcal{S}$) is given by
\begin{enumerate}
\item $\h_*\mathcal{E}(\mathcal{B}_a)(I_H,I_H)= \mathbb{Q}[0]\oplus\mathbb{Q}[1]$ with generators $[id_H]$ and $[m_1^H]$,
\item $\h_*\mathcal{E}(\mathcal{B}_a)(I_H,I_K)= 0$,
\item $\h_*\mathcal{E}(\mathcal{B}_a)(I_H,I_\mathbb{T})= \mathbb{Q}[0]$ with generator $[f_0^H]$,
\item $\h_*\mathcal{E}(\mathcal{B}_a)(I_\mathbb{T},I_H)= \mathbb{Q}[1]$ with generator $[\tilde{g}_1^H]$,
\item $\h_*\mathcal{E}(\mathcal{B}_a)(I_\mathbb{T},I_\mathbb{T})=(\oplus_{n \ge 0} \mathbb{Q}\mathcal{F}[2n+1])\oplus \mathbb{Q}[0]$ with generators $[i_{2n+1}^H]$ and $[id_\mathbb{T}]$.
\end{enumerate}
The nontrivial products and Massey products are
\begin{enumerate}
\item $[\tilde{g}_1^H][f_0^h]=[m_1^H]$
\item $[f_0^H][\tilde{g}_1^H]=[i_1^H]$
\item $\left< [f_0^H], [m_1^H],...,[m_1^H],[\tilde{g}_1^H]\right>= \{ [-i^H_{2n+1}] \}$ (where $[m^H_1]$ occurs $n$ times).
\end{enumerate}
\end{proposition}

Let us now turn to the other side of the Quillen equivalence.
One again shows that $\mathcal{E}(\mathcal{B}_t)$ is quasi-isomorphic to a more convenient  dga with many objects $\mathcal{E}_t$, which in turn is quasi-isomorphic to $\mathcal{S}$. Hence,
\[
\partial\mathcal{A} \simeq_Q \leftmod\mathcal{S} \simeq_Q \leftmod\mathcal{E}(\mathcal{B}_t) \simeq_Q \TMQ.
\]
We are going to follow the proof for the topological half of this equivalence for an arbitrary model $\mathcal{C}$ of $\Ho(\partial\mathcal{A})$ and see that we can apply the very same steps, which are 
\begin{itemize}
\item $\mathcal{E}(\mathcal{B}_t)$ is quasi-isomorphic to a dga with many objects $\mathcal{E}_t$ with specific nicer properties
\item $\mathcal{E}_t$ is quasi-isomorphic to $\mathcal{S}$.
\end{itemize}

\bigskip
To be more precise, let
\[
\Phi: \Ho(\partial \mathcal{A}) \longrightarrow \Ho(\mathcal{C})
\]
be an equivalence of triangulated categories. We obtain a set of small fibrant-cofibrant generators $\mathcal{X}$ of $\Ho(\mathcal{C})$ by taking fibrant-cofibrant replacements of $\Phi(\mathcal{B}_a)$. 

\bigskip
The homology data of $\mathcal{E}(\mathcal{X})$ is the same as for $\mathcal{E}(\mathcal{B}_a)$: \cite[Proposition 3.3]{Shi02} says that for fibrant and cofibrant $X,Y \in \mathcal{C}$,
\[
\h_*\Hom_{\mathcal{C}}(X,Y) \cong [X,Y]^{\Ho(\mathcal{C}))}
\]
as graded abelian groups and thus,
\[
\h_*\Hom_\mathcal{C}(\Phi(I),\Phi(J)) \cong [\Phi(I),\Phi(J)]^{\Ho(\mathcal{C}))} \cong [I,J]^{\Ho(dg\mathcal{A})} \cong \h_*\mathcal{E}(\mathcal{B}_a)(I,J).
\]
Since $\Phi$ is a functor, it preserves composition. Hence, the product structures in both homologies agree. Further, by \cite[Theorem A.3]{Shi02}, triangulated equivalences preserve all Toda brackets. Summarising this, we get the same result as earlier. Write $X_H$ for a fibrant-cofibrant replacement of $I_H$. 
\begin{enumerate}
\item $\h_*\mathcal{E}(\mathcal{X})(X_H,X_H)= \mathbb{Q}[0]\oplus\mathbb{Q}[1]$ with generators $[id_H]$ and $[M_1^H]$,
\item $\h_*\mathcal{E}(\mathcal{X})(X_H,X_K)= 0$,
\item $\h_*\mathcal{E}(\mathcal{X})(X_H,X_\mathbb{T})= \mathbb{Q}[0]$ with generator $[F_0^H]$,
\item $\h_*\mathcal{E}(\mathcal{X})(X_\mathbb{T},X_H)= \mathbb{Q}[1]$ with generator $[\tilde{G}_1^H]$,
\item $\h_*\mathcal{E}(\mathcal{X})(X_\mathbb{T},X_\mathbb{T})=(\oplus_{n \ge 0} \mathbb{Q}\mathcal{F}[2n+1])\oplus \mathbb{Q}[0]$ with generators $[T_{2n+1}^H]$ and $[id_\mathbb{T}]$.
\end{enumerate}
The nontrivial products and Massey products are
\begin{enumerate}
\item $[\tilde{G}_1^H][F_0^h]=[M_1^H]$
\item $[F_0^H][\tilde{G}_1^H]=[T_1^H]$
\item $\left< [F_0^H], [M_1^H],...,[M_1^H],[\tilde{G}_1^H]\right>= \{ [-T^H_{2n+1}] \}$ (where $[M^H_1]$ occurs $n$ times).
\end{enumerate}
Since the homology of $\mathcal{E}(\mathcal{X})$ is concentrated in non-negative degrees, we see that $\mathcal{E}(\mathcal{X})$ is quasi-isomorphic to its $-1$-connected cover. Using Postnikov approximations analogous to \cite[Proposition 5.7]{Shi02} we can modify this cover to obtain a dga of several objects $\mathcal{E}_X$ that is quasi-isomorphic to $\mathcal{E}(\mathcal{X})$, $(\mathcal{E}_X)_0 \cong H_0(\mathcal{E}_X)$, $(\mathcal{E}_X)_n=0$ for $n<0$, $\mathcal{E}_X(H,K)=0$ and $\mathcal{E}_X(H,\mathbb{T})$ is concentrated in degree zero. This will give us a technical advantage when constructing the quasi-isomorphism to $\mathcal{S}$. 

\begin{theorem}
The dgas with several objects $\mathcal{E}_X$ and $\mathcal{S}$ are quasi-isomorphic.
\end{theorem}

\begin{proof}
The proof is identical to the proof of \cite[Proposition 6.1]{Shi02} for the case $\mathcal{E}_t$, so we are only mentioning very little here. Shipley's proof only relies on combinatorics of properties preserved by the triangulated structure, mostly homology, multiplicative structure and Massey products.

The first step is to choose images for the generators of $\mathcal{S}$ This is done by sending a generator of $\mathcal{S}$ to a cycle in $\mathcal{E}_X$ representing a homology class in the right degree. Using multiplicative structure and Massey products, Shipley then proves by induction that this is indeed a well-defined homomorphism of dgas of many objects which induces an isomorphism in homology.
\end{proof}

\begin{corollary}
Let $\mathcal{C}$ be a stable model category with
\[
\Phi: \Ho(\partial \mathcal{A}) \longrightarrow \Ho(\mathcal{C})
\]
an equivalence of triangulated categories. Then $\mathcal{C}$ and $\partial \mathcal{A}$ (and hence $\TMQ$) are Quillen equivalent. Thus, $\Ho(\TMQ)$ is rigid.
\end{corollary}
\qed

The question of whether the rational $G$-equivariant stable homotopy category is rigid
for other for other compact Lie groups remains open. The above method would be 
hard to generalise as the number of Massey products increases rapidly as the rank of the group increases. 
One method to produce rigidity statements for general $G$ is discussed in the following section.

\section{Formality and Rigidity}\label{sec:formality}

We now turn to relating rigidity of triangulated categories to 
formality of differential graded algebras. 
We then examine how rigidity for finite and profinite groups
comes from a version of formality, whereas the proof
of rigidity for $S^1$ cannot be placed in this this framework. 
We end this section with another example 
of how intrinsic formality implies rigidity.

Consider a monogenic rational stable model category, $\mathcal{C}$, 
with $X$ a compact generator.
Then by the techniques above, $\mathcal{C}$ is Quillen equivalent to the category of 
modules of a differential graded algebra $A$, where the homology of $A$
is isomorphic to $\Ho(\mathcal{C})(X,X)$. 
Hence all of the homotopical information of  
$\mathcal{C}$ is contained in $A$.  
So we can ask how much of this information is 
contained in the homology of $A$?
This is analogous to asking whether $A$ can be recovered, 
up to quasi-isomorphism, from its homology. 
This is the question of whether $A$ is formal, which we define below.

\begin{definition}
A differential graded algebra $A$ is \textbf{formal} 
if it is quasi-isomorphic to its homology, 
regarded as a dga with trivial differential. 
A dga $B$ is said to be \textbf{intrinsically formal} if any other 
dga $C$ with $\h_*(B) \cong \h_*(C)$ is quasi-isomorphic to $B$. 
\end{definition}

Clearly, if a dga is intrinsically formal, then it is also formal, 
but the converse is not true. Note that any dga with a non-trivial 
Massey product cannot be formal. These notions were originally introduced
in the context of rational homotopy theory, see \cite{Sull77} or \cite{FHT01}. 
However they have also been of use in more algebraic settings, \cite{RW11}.

We can rephrase intrinsic formality as follows. A graded algebra $R$ is 
intrinsically formal if whenever a dga $B$ has $\h_*(B) \cong R$, 
then $R$ is quasi-isomorphic to $B$. Hence intrinsic formality
is really a property of graded algebras, whereas formality is 
a property of dgas.

We can adjust this to the commutative setting, but here we must take care
to make sure that we only consider zig-zag of quasi-isomorphisms that pass through commutative dgas at all stages. 

\begin{definition}
A commutative differential graded algebra $A$ is \textbf{formal as a commutative dga}, if it is quasi-isomorphic, as a commutative dga, to its homology regarded as a dga with trivial differential. 
A commutative dga $B$ is \textbf{intrinsically formal as a commutative dga} if 
any other commutative dga $C$ with $\h_*(B) \cong \h_*(C)$ is quasi-isomorphic to $B$ as a commutative dga. 
\end{definition}

We provide some of examples of intrinsic formality.
In each case our ground ring will be a field $k$
of characteristic zero. 

\begin{ex}
Any graded algebra with homology concentrated in degree 
zero is intrinsically formal. Let $A$ be such a dga and let  $C_0 A$ be its  
$(-1)$-connective cover. This dga is defined 
as follows, $(C_0 A)_k$ is zero for $k <0$, $A_k$
for $k >0$ and for $k=0$ is the zero-cycles of $A$. 
There inclusion $C_0 A \to A$ 
is a quasi-isomorphism. 
We also have a quotient map 
$C_0 A \to \h_* (A)$ which in degree $0$
sends the cycle $z$ to the homology class of $[z]$, 
in all other degrees it is the zero map.
\end{ex}

\begin{ex}
Any free graded algebra is intrinsically formal. 
Let $A$ be a dga whose homology is isomorphic to the 
free associative algebra on generators $x_i$. 
Then by choosing cycle representatives 
for each $x_i$, we obtain a quasi-isomorphism
$\h_* A \to A$.
\end{ex}

\begin{ex}
Products of associative algebras preserve intrinsic formality. 
This follows since for any dga $A$, $A \times -$ preserves quasi-isomorphisms.
\end{ex}

Now we turn our attention to the commutative case. Our first example
follow by the same arguments as above. 

\begin{ex}
Any commutative graded algebra with homology in degree zero is intrinsically formal. 
Any free commutative graded algebra is intrinsically formal. 
\end{ex}

\begin{ex}
Products preserve intrinsic formality of commutative dgas, as do tensor products over $k$. 
Both of these statements follow from the fact that for any dga $A$, both 
$A \times -$ and $A \otimes -$ preserve quasi-isomorphisms.
\end{ex}

\begin{ex}
Let $R= k[x_1, \dots x_n]$, then $R$ modulo
a regular sequence $(r_1, \dots r_m)$ is intrinsically formal.
Let $A$ be a commutative dga, with $R \cong \h_*(A)$.
Then by choosing cycles representatives for the $x_i$,
we have a map $f : k[x_1, \dots x_n] \to A$. 

Consider the commutative dga $R_m$, the underlying algebra of which 
takes form 
$k[x_1, \dots x_n] \otimes \Lambda[y_1, \dots y_m]$, with differential defined by 
$\partial y_i = r_i$ and $\partial x_i = 0$.  
Since 
the sequence is regular, $\h_*(R_m) =  R/(r_1, \dots r_m)$.
We define $f_m : R_m \to A$ by sending
$x_i$ to $f(x_i)$ and $y_i$ to $f(r_i)$. 
We define quasi-isomorphisms $g_m : R_m \to R/(r_1, \dots r_m)$ by sending
$x_i$ to $x_i$ and $y_i$ to zero.  
\end{ex}

We now describe the relation between the formality and rigidity,
 which is most easily seen when we work rationally. Consider a 
cofibrantly generated, proper, simplicial, rational stable model category $\mathcal{C}$ with a single compact generator $X$. Then by 
the results of \cite{shi07} (see Theorem \ref{thm:rationalmorita}),  
$\mathcal{C}$ is Quillen equivalent to $\mathcal{E}_\mathbb{Q}(X)$-modules 
in rational chain complexes for some dga 
$\mathcal{E}_\mathbb{Q}(X)$ satisfying
\[
\h_*(\mathcal{E}_\mathbb{Q}(X)) \cong [X,X]^\mathcal{C}_* \cong \pi_* \mathcal{E}(X)
\]
Note that if $\mathcal{E}(X)$
is commutative, then we can also assume that 
$\mathcal{E}_\mathbb{Q}(X)$ is a commutative dga. 
This happens for example when $X$ is the unit of a monoidal model category.
In a related context, the work of \cite{GS11} fundamentally depends upon 
commutativity of terms like $[X,X]^\mathcal{C}_*$.

If $\mathcal{E}_\mathbb{Q}(X)$ is formal, then we would know that 
$\mathcal{C}$ is Quillen equivalent to $\pi_*(\mathcal{E}(X))$-modules 
in rational chain complexes. But we have little control over the dga $\mathcal{E}_\mathbb{Q}(X)$. We only have a good understanding of its homology. 
So we assume a stronger condition, that it is intrinsically formal 
(or rather, that the graded algebra $[X,X]_*^{\mathcal{C}}$ is intrinsically formal).
We thus achieve the same conclusion, that
$\mathcal{C}$ is Quillen equivalent to $\pi_*(\mathcal{E}(X))$-modules, 
but the assumptions are 
considerably easier to verify.

To link this to rigidity, let $\mathcal{D}$ be another 
cofibrantly generated, proper, simplicial, stable model category with 
$\Ho(\mathcal{D})$ triangulated equivalent to $\Ho(\mathcal{C})$. Then 
$\mathcal{D}$ is also necessarily rational and it must have a generator
$Y$, which is also compact. Furthermore $[Y,Y]^\mathcal{D}_* \cong [X,X]^\mathcal{C}_*$
as dgas and so $[Y,Y]^\mathcal{D}_*$ is intrinsically formal. 
Hence, $\mathcal{D}$ is Quillen equivalent to 
$[Y,Y]^\mathcal{D}_*$-modules in rational chain complexes. 
A quasi-isomorphism of differential graded algebras induces
a Quillen equivalence on categories of modules, so it follows that
$\mathcal{C}$ and $\mathcal{D}$ are Quillen 
equivalent. We have thus proven that $\mathcal{C}$ is rigid. 
We can think of this as the statement that 
intrinsic formality implies rigidity. 

\begin{theorem}
Let $\mathcal{C}$ be a rational, monogenic stable 
model category, that is also 
proper, cofibrantly generated and simplicial.
Let $X$ be a generator and assume that the graded ring
$\Ho(\mathcal{C})(X,X)$ is intrinsically formal. 
Then if $\mathcal{D}$ is another stable model category
that is also proper, cofibrantly generated and simplicial, 
such that $\Ho(\mathcal{C})$ and $\Ho(\mathcal{D})$
are triangulated equivalent. 
Then $\mathcal{C}$ and $\mathcal{D}$ are Quillen equivalent. 
\end{theorem}

We can extend the notion of formality to dgas with many objects, but 
examples of intrinsically formal dgas with many objects are hard to find
in general. We know of two such examples, both from equivariant homotopy theory. 
The first is where the homology of the dga with many objects is 
concentrated in degree zero.
The homology of such a dga with many objects is intrinsically formal 
via a multiple-object version of the above examples. 
Hence the earlier rigidity statements for 
finite and profinite groups are examples of 
where intrinsic formality implies rigidity. 

The second example comes from the work of Greenlees and Shipley 
\cite{GS11}. That paper uses a version of intrinsic formality of 
diagrams of commutative dgas. Note that such a diagram can be written
as a dga with many objects.In that paper, the authors 
show that the homotopy theory of 
rational $S^1$-equivariant spectra is captured by the diagram 
of commutative dgas 
\[
T=\left( \mathbb{Q} \longrightarrow 
\mathcal{E}^{-1} \mathcal{O}_\mathcal{F}
\longleftarrow 
\mathcal{O}_\mathcal{F} \right).
\]
They then prove that this diagram is intrinsically formal 
amongst diagrams of commutative dgas 
with shape $\bullet \to \bullet \leftarrow \bullet$. 
That is, if $R=(A \to B \leftarrow C)$ is a diagram of commutative dgas, whose homology
is isomorphic to $T$, then 
$R$ and $T$ are quasi-isomorphic as diagrams of commutative dgas. 
Commutativity is essential for this proof. 

We remark that our proof that the rational $S^1$-equivariant stable
homotopy category is rigid does not rely on an intrinsic 
formality statement. We would like to use the formality statement of 
Greenlees and Shipley as the basis of a rigidity proof, 
but it is not clear how to realise this diagram of dgas as an artefact 
of the triangulated category. Furthermore, it would be difficult to generalise the 
construction of the numerous different model categories of \cite{GS11} needed to 
relate $\mathbb{T}$-spectra to this diagram of dgas. 

We now give an example of rational equivariant rigidity 
that uses our formality discussion above, but does not rely
on all homology being concentrated in degree zero. 

\begin{ex}
Consider the group $S^1$ and its universal free space 
$\E S^1$. By adding a disjoint basepoint and 
taking the suspension spectrum we obtain a $S^1$-equivariant
spectrum, $\E \mathbb{T}_+$. Then consider the 
function spectrum $\DE \mathbb{T}_+ = F(\E \mathbb{T}_+, \mathbb{S}_\mathbb{Q})$,
where $\mathbb{S}_\mathbb{Q}$ is the rationalised sphere spectrum.
The spectrum $\DE \mathbb{T}_+$ is a commutative ring spectrum
using the diagonal map 
$\E \mathbb{T}_+ \to \E \mathbb{T}_+ \wedge \E \mathbb{T}_+$.
We also have a non-equivariant commutative ring spectrum
$\Dclassify \mathbb{T}_+ = F(\classify \mathbb{T}_+, \mathbb{S}_\mathbb{Q})$,
where we use the non-equivariant rationalised sphere spectrum in this 
construction.

Consider the model category of $\DE \mathbb{T}_+$-modules in rational 
$\mathbb{T}$-spectra and the model category of $\Dclassify \mathbb{T}_+$-modules in non-equivariant rational spectra.
By \cite[Example 11.4]{GS11} we know that these model categories 
are Quillen equivalent. 
These two model categories are important to the study of stable 
$\mathbb{T}$-equivariant phenomena, indeed this 
Quillen equivalence is the prototype for the work of Greenlees
and Shipley on torus-equivariant spectra.

Now we can apply the above discussion of formality. 
The homotopy groups of $\Dclassify \mathbb{T}_+$ 
are given by the ring $\mathbb{Q}[c]$, with $c$ of degree $-2$.
Since this graded commutative ring is intrinsically 
formal, we can conclude that the category
of $\Ho(\leftmod \DE \mathbb{T}_+)$ is rigid. 
In particular, any underlying model category is Quillen 
equivalent to the model category 
of differential graded $\mathbb{Q}[c]$-modules. 
\end{ex}

\bibliography{ERliteratur}
\bibliographystyle{alpha}

\end{document}